\documentclass{amsart}
\usepackage[utf8]{inputenc}
\usepackage{amsmath,amsthm,amssymb,url}
\usepackage[margin = 2.5cm]{geometry}
\usepackage{tikz}
\usepackage{ytableau}
\usepackage{color}
\newtheorem{theorem}{Theorem}

\newtheorem{prop}[theorem]{Proposition}

\newtheorem{lemma}[theorem]{Lemma}
\newtheorem{definition}[theorem]{Definition}

\theoremstyle{definition}

\newtheorem{example}[theorem]{Example}

\newcommand{\rd}[1]{{\color{red} #1}}

\def\cprime{$'$}

\newcommand{\p}[1]{\mathcal #1}

\DeclareMathOperator{\SYT}{SYT}
\newcommand{\Z}{{\mathbb Z}}

\title{
Hook, line and sinker: a bijective proof of the skew shifted hook-length formula}

\author{Matja\v z Konvalinka}%
\address{Faculty of Mathematics and Physics, University of Ljubljana, and Institute of Mathematics, Physics and Mechanics, Ljubljana, Slovenia}
\urladdr{http://www.fmf.uni-lj.si/~konvalinka/}
\thanks{The author acknowledges the financial support from the Slovenian Research Agency (research core funding No. P1-0294).}

\date{\today}

\begin{document}

\begin{abstract}
A few years ago, Naruse presented a beautiful cancellation-free hook-length formula for skew shapes, both straight and shifted. The formula involves a sum over objects called \emph{excited diagrams}, and the term corresponding to each excited diagram has hook lengths in the denominator, like the classical hook-length formula due to Frame, Robinson and Thrall.\\
Recently, the formula for skew straight shapes was proved via a simple bumping algorithm. The aim of this paper is to extend this result to skew shifted shapes. Since straight skew shapes are special cases of skew shifted shapes, this is a bijection that proves the whole family of hook-length formulas, and is also the simplest known bijective proof for shifted (non-skew) shapes. A weighted generalization of Naruse's formula is also presented.
\end{abstract}

\maketitle

\section{Introduction} \label{intro}

The celebrated \emph{hook-length formula} gives an elegant product expression for the number of standard Young tableaux of fixed shape $\lambda$:

$$f^\lambda = \frac{|\lambda|!}{\prod_{u \in [\lambda]} h(u)}$$

\medskip

Here $h(u)$ is the hook length of the cell $u$.

\medskip

The formula gives dimensions of irreducible representations of the symmetric group. The formula was discovered by Frame,
Robinson and Thrall in~\cite{FRT} based on earlier results of Young \cite{You}, Frobenius \cite{Fro} and Thrall \cite{Thr}.  Since then, it has been reproved, generalized and extended in several different ways, and applied in a number of fields ranging from algebraic geometry to probability, and from group theory to the analysis of algorithms.

\medskip

Interestingly, an identical formula (when hook lengths are defined appropriately) also holds for strict partitions and standard Young tableaux of \emph{shifted} shape. Even though it was discovered before the formula for straight shapes \cite{Thr}, it is typically considered the ``lesser'' of the two: less well known, less interesting, and with more complicated proofs. Shortly after the famous Greene-Nijenhuis-Wilf probabilistic proof of the ordinary hook length formula \cite{gnw}, Sagan \cite{sagan} extended the argument to the shifted case. The proof needs a careful analysis of special cases and a delicate double induction, and therefore lacks the intuitiveness of the original hook-walk proof. In 1995, Krattenthaler \cite{krat1} provided a bijective proof. While short, it is very involved, as it needs a variant of Hillman-Grassl algorithm, a bijection that comes from Stanley's $(P,\omega)$-partition theory, and the involution principle of Garsia and Milne. A few years later, Fischer \cite{fischer} gave the most direct proof of the formula, in the spirit of Novelli-Pak-Stoyanovskii's \cite{nps} bijective proof of the ordinary hook-length formula. At almost 50 pages in length, the proof is very involved. Bandlow \cite{Ban} gave a short proof via interpolation, and there is a variant of the hook-walk proof in \cite{kon:shifted}; again, special cases need to be considered, and the bijection is hard to describe succinctly. There are also many generalizations of both formulas, such as the $q$-version of Kerov~\cite{Ker1}, and its further generalizations and variations (see~\cite{GH,Ker2} and also \cite{ckp}). There are also a great number of proofs of the more general Stanley's hook-content formula (see e.g.~\cite[Corollary 7.21.4]{ec2}), see for example \cite{rw, krat2, krat3}. 

\medskip

There is no (known) product formula for the number of standard Young tableaux of skew shape (straight or shifted), even though some formulas have been known for a long time. In 2014, Hiroshi Naruse \cite{naru} presented and outlined a proof of a remarkable cancellation-free generalization for skew shapes, both straight and shifted. Here we present the two formulas for shifted shapes.

\medskip

The \emph{shifted diagram of type B}, $[\lambda]^B$, of a strict partition $\lambda$ (i.e., a partition with distinct parts) is the set of cells $\{(i,j) \colon 1 \leq i \leq \ell(\lambda), i \leq j \leq i + \lambda_i - 1\}$. The \emph{shifted diagram of type D}, $[\lambda]^D$, of a strict partition $\lambda$ is the set of cells $\{(i,j) \colon 1 \leq i \leq \ell(\lambda), i + 1 \leq j \leq i + \lambda_i\}$.

\medskip

An \emph{excited move of type B} is the move of a cell $(i,j)$ of a shifted diagram to position $(i+1,j+1)$, provided that the cells $(i+1,j)$, $(i,j+1)$ and $(i+1,j+1)$ are not in the diagram. An \emph{excited move of type D}  is the move of:
\begin{itemize}
 \item a \emph{non-diagonal} cell $(i,j)$, $i + 1 < j$, of a shifted diagram to position $(i+1,j+1)$, provided that the cells $(i+1,j)$, $(i,j+1)$ and $(i+1,j+1)$ are not in the diagram, or
 \item a \emph{diagonal} $(i,i+1)$ cell of a shifted diagram diagonally to position $(i+2,i+3)$, provided that the cells $(i,i+2)$, $(i,i+3)$, $(i+1,i+2)$, $(i+1,i+3)$ and $(i+2,i+3)$ are not in the diagram.
\end{itemize}

Let $\p E^B(\lambda/\mu)$ (respectively, $\p E^D(\lambda/\mu)$) denote the set of all \emph{excited diagrams} of shifted shape $\lambda/\mu$, diagrams in $[\lambda]^B$ (resp., $[\lambda]^D$) obtained by taking the diagram of $\mu$ and performing series of excited moves of type B (resp., type D) in all possible ways. They were introduced by Ikeda and Naruse \cite{IN1}.

\medskip

Naruse's formula says that

\begin{align}
 f^{\lambda/\mu} &= |\lambda/\mu|! \sum_{D \in \p E^B(\lambda/\mu)} \prod_{u \in [\lambda]^B \setminus D} \frac 1{h^B(u)} = \frac{|\lambda/\mu|!}{\prod_{u \in [\lambda]^B } h^B(u)} \sum_{D \in \p E^B(\lambda/\mu)} \prod_{u \in  D} h^B(u)  \label{eqB} \\
 &= |\lambda/\mu|! \sum_{D \in \p E^D(\lambda/\mu)} \prod_{u \in [\lambda]^D \setminus D} \frac 1{h^D(u)} = \frac{|\lambda/\mu|!}{\prod_{u \in [\lambda]^D } h^D(u)} \sum_{D \in \p E^D(\lambda/\mu)} \prod_{u \in  D} h^D(u)  \label{eqD}
\end{align}

where all the hook lengths are evaluated in the diagram of $\lambda$.

\medskip

Our main result (Theorem \ref{thm1}) are the following formulas, valid for strict partitions $\lambda$, $\mu$ and for commutative variables $x_i$:

\begin{align}
\left( \sum_{k \in \p W^B(\mu,\lambda)} x_k \right) \sum_{D \in \p E^B(\lambda/\mu)} \prod_{(i,i) \in D} x_i \prod_{\substack{(i,j) \in D \\ i < j}} (x_i + x_j) &= \sum_{\mu \lessdot \nu \subseteq \lambda} \sum_{D \in \p E^B(\lambda/\nu)} \prod_{(i,i) \in D} x_i \prod_{\substack{(i,j) \in D \\ i < j}} (x_i + x_j) ,  \label{eqmainB} \\
\left( \sum_{k \in \p W^D(\mu,\lambda)} x_k \right) \sum_{D \in \p E^D(\lambda/\mu)} \prod_{(i,j) \in D} (x_i + x_j) &= \sum_{\mu \lessdot \nu \subseteq \lambda} \sum_{D \in \p E^D(\lambda/\nu)} \prod_{(i,j) \in D} (x_i + x_j),  \label{eqmainD}
\end{align}
where $\p W^B(\mu,\lambda)$ and $\p W^D(\mu,\lambda)$ are certain finite subsets of positive integers. The formulas specialize to recursive versions of equations \eqref{eqB} and \eqref{eqD}.

\medskip

Like in \cite{kon:naruse}, the proof of both formulas is bijective and uses a simple bumping algorithm. Moreover, note that:
\begin{itemize}
 \item if $\mu$ is the staircase shape $(\ell(\lambda),\ldots,1)$, the formulas \eqref{eqB} and \eqref{eqD} become the classical hook-length formula for straight (non-skew, non-shifted) shapes
 \item if $\mu$ is, more generally, a strict partition of the same length as $\lambda$, the formulas \eqref{eqB} and \eqref{eqD} become the hook-length formula for (straight) skew shapes (also due to Naruse)
 \item if $\mu = \emptyset$, the formulas \eqref{eqB} and \eqref{eqD} become the (also classical) hook-length formula for shifted shapes
\end{itemize}

Our algorithms thus also give bijective proofs of all these five important formulas. Let us emphasize that the non-skew cases (i.e, non-shifted and shifted) are not related in the sense that neither formula implies the other.

\medskip

The fact that the algorithm is so easy to describe (it takes less than a page!) is quite remarkable. One would hope that this fact will help the shifted hook-length formula overcome its ``lesser'' status; indeed, how can it be lesser if it is a straightforward generalization with a bijective proof that is not more complicated in any way?

\medskip

Of course, simplicity of definition does not imply that an algorithm is computationally efficient. We prove that the number of steps needed is $2^{\Omega(\sqrt n)}$ for partitions $\lambda,\mu$ of size at most $n$; see Theorem \ref{complexity}.

\medskip

This is a natural sequel to \cite{kon:naruse} and will only sketch some of the proofs that are similar. A reader who is familiar with \cite{kon:naruse} should note the following differences:
\begin{itemize}
 \item there is only one set of variables, not two, see Definition \ref{xi},
 \item the linear factors on the left of equations \eqref{eqmainB} and \eqref{eqmainD} are slightly different from the straight-shape case
 \item in \eqref{eqmainB}, there are monomials appearing in the product (not just binomials)
 \item the insertion algorithm always starts from the top row
\end{itemize}

\medskip

In Section \ref{def}, we give basic definitions and notation. In Section \ref{identity}, we motivate equation \eqref{eqmainB} and \eqref{eqmainD} and show how they imply \eqref{eqB} and \eqref{eqD}. In Section \ref{bijection}, we use versions of the bumping algorithm on tableaux to prove the identities bijectively. In Section \ref{weighted}, we show weighted versions of formulas \eqref{eqB} and \eqref{eqD}. In Section \ref{proofs}, we show examples and sketch the proofs of the technical statements from Section \ref{bijection}.

\section{Basic definitions and notation} \label{def}

A \emph{partition} is a weakly decreasing finite sequence of positive integers $\lambda = (\lambda_1,\lambda_2,\ldots,\lambda_\ell)$. We say that $\lambda$ is a \emph{strict partition} if the sequence is strictly decreasing. We call $|\lambda| = \lambda_1 + \cdots + \lambda_\ell $ the \emph{size} of $\lambda$ and $\ell = \ell(\lambda)$ the \emph{length} of $\lambda$. We write $\lambda_i = 0$ for $i > \ell(\lambda)$. The \emph{shifted diagrams of types B and D} of $\lambda$ are
$$[\lambda]^B = \{(i,j) \colon 1 \leq i \leq \ell(\lambda), i \leq j \leq \lambda_i + i - 1\}$$
and
$$[\lambda]^D = \{(i,j) \colon 1 \leq i \leq \ell(\lambda), i < j \leq \lambda_i + i\}.$$
Note that the difference is that the columns start with $2$ in type D.

\medskip

We call the elements of $[\lambda]^B$ and $[\lambda]^D$ the \emph{cells} of $\lambda$. For partitions $\mu$ and $\lambda$, we say that $\mu$ is \emph{contained} in $\lambda$, $\mu \subseteq \lambda$, if $[\mu]^B \subseteq [\lambda]^B$ (which is equivalent to $[\mu]^D \subseteq [\lambda]^D$). In that case, we say that $\lambda/\mu$
is a \emph{skew shape} of size $|\lambda/\mu| = |\lambda| - |\mu|$, and the diagrams of $\lambda/\mu$ are $[\lambda/\mu]^B = [\lambda]^B \setminus [\mu]^B$ and $[\lambda/\mu]^D = [\lambda]^D \setminus [\mu]^D$. We write $\mu \lessdot \lambda$ if $\mu \subseteq \lambda$
and $|\lambda/\mu| = 1$. In this case, we also say that $\lambda$ \emph{covers} $\mu$. Note that diagrams and cover relations are different than in the context of (not necessarily strict) partitions, while containment is the same.

\medskip

We often represent a strict partition $\lambda$ graphically, with a cell $(i,j)$ in $[\lambda]^B$ or $[\lambda]^D$ represented by a unit square in position $(i,j)$. In this paper, we use English notation, so for example the Young diagram of the partition $\lambda = (6,5,3,2)$ is

\medskip

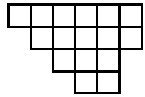
\begin{center}
\ytableausetup{smalltableaux}
\begin{ytableau}
{} & {} & {} & {} & {} & {} \\ \none  & {} & {} & {} & {} & {} \\ \none & \none  & {} & {} & {} \\  \none & \none & \none  & {} & {}
\end{ytableau}

\end{center}

\medskip

We often omit parenthesis and commas, so we could write $\lambda = 6532$. Also note that the cells appear in rows $1,\ldots,\ell(\lambda)$ and either in columns $1,\ldots,\lambda_1$ or in columns $2,\ldots,\lambda_1+1$, depending on the type we are interested in.

\medskip

A \emph{corner} of $\lambda$ is a cell that can be removed from $[\lambda]^B$, i.e., a cell $(i,j) \in [\lambda]^B$ satisfying $(i+1,j),(i,j+1) \notin [\lambda]^B$. An \emph{outer corner} of $\lambda$ is a cell that can be added to $[\lambda]^B$, i.e., a cell $(i,j) \notin [\lambda]^B$ satisfying $i = 1$ or $(i-1,j) \in [\lambda]^B$, and $j = 1$
 or $(i,j-1) \in [\lambda]^B$. The partition $6532$ has (in type B) corners $(2,6)$ and $(4,5)$, and outer corners $(1,7)$, $(3,6)$ and $(5,5)$. In type D, add $1$ to the second coordinate in all cases.

\medskip

For a strict partition $\lambda$, denote by $\lambda_i'$ the number of cells in the $i$-th column of $[\lambda]^B$. For $\lambda = 6532$, we get $\lambda'_i = i$ for $i = 1,\ldots,4$, $\lambda'_5=4$, $\lambda'_6 = 2$, $\lambda'_i = 0$ for $i \geq 7$. Let us emphasize that $\lambda_i'$ retains the same meaning when we are in type D, so the number of cells in the $i$-th column of $[\lambda]^D$ is $\lambda'_{i-1}$.

\medskip

The \emph{type B hook length} of the cell $(i,j) \in [\lambda]^B$, $h^B(i,j)$, is defined as:
\begin{itemize}
 \item the number of cells in row $i$ if $i = j$
 \item the number of cells in rows $i$ and $j$ if $i < j \leq \ell(\lambda)$
 \item the number of cells weakly to the right of $(i,j)$ and below $(i,j)$ if $i < j$ and $j > \ell(\lambda)$
\end{itemize}
The definition for $i < j \leq \ell(\lambda)$ is often phrased as ``the number of cells weakly to the right of or below  $(i,j)$, and in row $j$, where the cell $(j,j)$ is counted twice''. For example, the hook lengths of the cells $(2,2)$, $(1,4)$ and $(2,5)$ in $[6532]^B$ are $5$, $8$ and $4$:

\medskip

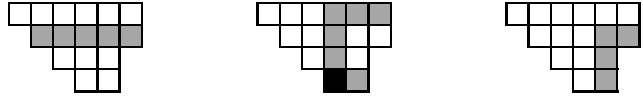
\begin{center}
\ytableausetup{smalltableaux}
 \begin{ytableau}
{} & {} & {} & {} & {} & {} \\ \none  & *(gray!70) & *(gray!70) & *(gray!70) & *(gray!70) & *(gray!70) \\ \none & \none  &  & {} & {} \\  \none & \none & \none  &  &
\end{ytableau} \qquad \qquad
\begin{ytableau}
{} & & & *(gray!70) & *(gray!70) & *(gray!70) \\ \none  & {} &  &  *(gray!70) & & {} \\ \none & \none &  & *(gray!70) & {} \\  \none & \none & \none  & *(black) & *(gray!70)
\end{ytableau} \qquad \qquad
 \begin{ytableau}
{} & {} & {} & {} & {} & {} \\ \none  & {} & {} & & *(gray!70) & *(gray!70) \\ \none & \none  & {} &  & *(gray!70) \\  \none & \none & \none  & & *(gray!70)
\end{ytableau}
\end{center}

\medskip

The black square is counted twice.

\medskip

The \emph{type D hook length} of the cell $(i,j) \in [\lambda]^D$ is defined as the number of cells weakly to the right of $(i,j)$, below $(i,j)$, and in row $j$. For example, the hook lengths of the cells $(1,4)$ and $(2,5)$ in  $[6532]^D$ are $8$ and $5$:

\medskip

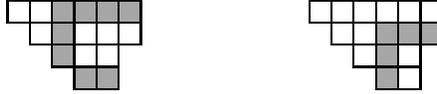
\begin{center}
\ytableausetup{smalltableaux}
 \begin{ytableau}
{} & {} & *(gray!70) & *(gray!70) & *(gray!70) & *(gray!70) \\ \none  & {} & *(gray!70) & {} & {} & {} \\ \none & \none  & *(gray!70) & {} & {} \\  \none & \none & \none  & *(gray!70) & *(gray!70)
\end{ytableau} \qquad \qquad \qquad
 \begin{ytableau}
{} & {} & {} & {} & {} & {} \\ \none  & {} & {} & *(gray!70) & *(gray!70) & *(gray!70) \\ \none & \none  & {} & *(gray!70) & {} \\  \none & \none & \none  & *(gray!70) & {}
\end{ytableau}

\end{center}

\medskip

In both cases, the \emph{B/D hook} is the set of cells used to define the hook length.

\medskip

A \emph{standard Young tableau} (or \emph{SYT} for short) of shifted shape $\lambda$ is a bijective map $T$ from the diagram of $\lambda$ to $\{1,\ldots,|\lambda|\}$, $(i,j) \mapsto T_{ij}$, satisfying $T_{ij} < T_{i,j+1}$ if $(i,j),(i,j+1)$ are in the diagram and $T_{ij} < T_{i+1,j}$ if $(i,j),(i+1,j)$ are in the diagram. The number of SYT's of shifted shape $\lambda$ is denoted by $f^\lambda$, and the set of standard Young tableaux of shape
$\lambda/\mu$ by $\SYT(\lambda/\mu)$. The type does not matter in this definition. The following illustrates $f^{42} = 5$:

\medskip

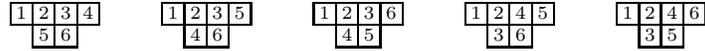
\begin{center}
\ytableausetup{centertableaux}
\ytableausetup{smalltableaux}
\begin{ytableau}
1 & 2 & 3 & 4\\
\none & 5 & 6
\end{ytableau}
\qquad
\begin{ytableau}
1 & 2 & 3 & 5 \\
\none & 4 & 6
\end{ytableau}
\qquad
\begin{ytableau}
1 & 2 & 3 & 6 \\
\none & 4 & 5
\end{ytableau}
\qquad
\begin{ytableau}
1 & 2 & 4 & 5 \\
\none & 3 & 6
\end{ytableau}
\qquad
\begin{ytableau}
1 & 2 & 4 & 6 \\
\none & 3 & 5
\end{ytableau}
\end{center}

\medskip

The \emph{shifted hook-length formula} gives a product expression for the number of standard Young tableaux:

$$f^\lambda = \frac{|\lambda|!}{\prod_{u \in [\lambda]} h^B(u)} = \frac{|\lambda|!}{\prod_{u \in [\lambda]} h^D(u)}$$

The formulas are equivalent: one multiset of hook lengths is a permutation of the other. For example, $f^{42} = \frac{6!}{4 \cdot 6 \cdot 3 \cdot 1 \cdot 2 \cdot 1} = \frac{6!}{6 \cdot 4 \cdot 3 \cdot 1 \cdot 2 \cdot 1} = 5$.

\medskip

We can define a standard Young tableau of skew shifted shape $\lambda/\mu$ analogously. The number of SYT's of shifted shape $\lambda/\mu$ is denoted by $f^{\lambda/\mu}$. The following illustrates $f^{543/2} = 12$:

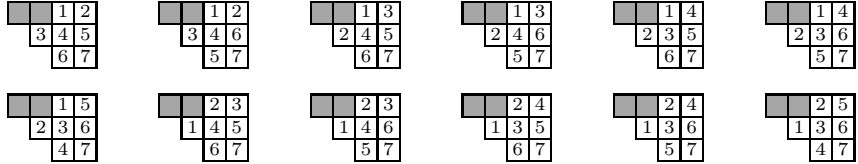
\begin{center}
\ytableausetup{centertableaux}
\ytableausetup{smalltableaux}

\begin{ytableau}
 *(gray!70) & *(gray!70) & 1 & 2 \\
\none & 3 & 4 & 5 \\
\none & \none & 6 & 7
\end{ytableau}
\qquad
\begin{ytableau}
 *(gray!70) & *(gray!70) & 1 & 2 \\
\none & 3 & 4 & 6 \\
\none & \none & 5 & 7
\end{ytableau}
\qquad
\begin{ytableau}
 *(gray!70) & *(gray!70) & 1 & 3 \\
\none & 2 & 4 & 5 \\
\none & \none & 6 & 7
\end{ytableau}
\qquad
\begin{ytableau}
 *(gray!70) & *(gray!70) & 1 & 3 \\
\none & 2 & 4 & 6 \\
\none & \none & 5 & 7
\end{ytableau}
\qquad
\begin{ytableau}
 *(gray!70) & *(gray!70) & 1 & 4 \\
\none & 2 & 3 & 5 \\
\none & \none & 6 & 7
\end{ytableau}
\qquad
\begin{ytableau}
 *(gray!70) & *(gray!70) & 1 & 4 \\
\none & 2 & 3 & 6 \\
\none & \none & 5 & 7
\end{ytableau}

\vspace{0.3cm}

\begin{ytableau}
 *(gray!70) & *(gray!70) & 1 & 5 \\
\none & 2 & 3 & 6 \\
\none & \none & 4 & 7
\end{ytableau}
\qquad
\begin{ytableau}
 *(gray!70) & *(gray!70) & 2 & 3 \\
\none & 1 & 4 & 5 \\
\none & \none & 6 & 7
\end{ytableau}
\qquad
\begin{ytableau}
 *(gray!70) & *(gray!70) & 2 & 3 \\
\none & 1 & 4 & 6 \\
\none & \none & 5 & 7
\end{ytableau}
\qquad
\begin{ytableau}
 *(gray!70) & *(gray!70) & 2 & 4 \\
\none & 1 & 3 & 5 \\
\none & \none & 6 & 7
\end{ytableau}
\qquad
\begin{ytableau}
 *(gray!70) & *(gray!70) & 2 & 4 \\
\none & 1 & 3 & 6 \\
\none & \none & 5 & 7
\end{ytableau}
\qquad
\begin{ytableau}
 *(gray!70) & *(gray!70) & 2 & 5 \\
\none & 1 & 3 & 6 \\
\none & \none & 4 & 7
\end{ytableau}

\end{center}

\medskip

Suppose that $D \subseteq [\lambda]^B$. If $(i,j) \in D$ and $(i+1,j),(i,j+1),(i+1,j+1) \in [\lambda]^B \setminus D$, then an \emph{excited move of type B} with respect to $\lambda$ is the replacement of $D$ with $D' = D \setminus \{(i,j)\} \cup \{(i+1,j+1)\}$. For strict partitions $\lambda,\mu$, then an \emph{excited diagram} of type B and shape $\lambda/\mu$ is a diagram contained in $[\lambda]^B$ that can be obtained from $[\mu]^B$ with a series of excited moves. Let $\p E^B(\lambda/\mu)$ denote the (empty unless $\mu \subseteq \lambda$) set of all excited diagrams of type B and shape $\lambda/\mu$. The following shows $\p E^B(432/2)$:

\medskip

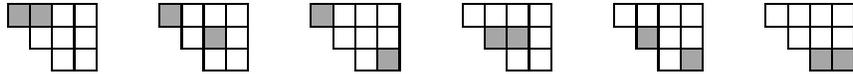
\begin{center}
\ytableausetup{smalltableaux}
 \begin{ytableau}
 *(gray!70) & *(gray!70) & & \\
  \none & & & \\
  \none & \none & &\\
\end{ytableau}
\qquad
\begin{ytableau}
 *(gray!70) &  & & \\
  \none & & *(gray!70) & \\
  \none & \none & &\\
\end{ytableau}
\qquad
\begin{ytableau}
 *(gray!70) & & & \\
  \none & & & \\
  \none & \none & & *(gray!70)\\
\end{ytableau}
\qquad
\begin{ytableau}
 {} & & & \\
  \none & *(gray!70) & *(gray!70) & \\
  \none & \none & & \\
\end{ytableau}
\qquad
\begin{ytableau}
 {} & & & \\
  \none &*(gray!70) & & \\
  \none & \none & & *(gray!70)\\
\end{ytableau}
\qquad
\begin{ytableau}
{} & & & \\
  \none & & & \\
  \none & \none & *(gray!70) & *(gray!70)\\
\end{ytableau}
\end{center}

\medskip

If $(i,j) \in D$ for $i + 1 < j$, $(i+1,j),(i,j+1),(i+1,j+1) \in [\lambda]^D \setminus D$, then an \emph{excited move of type D} with respect to $\lambda$ is the replacement of $D$ with $D' = D \setminus \{(i,j)\} \cup \{(i+1,j+1)\}$. On the other hand, if $(i,i+1) \in D$, $(i,i+2),(i,i+3),(i+1,i+2),(i+1,i+3),(i+2,i+3) \in [\lambda]^D \setminus D$, an excited move of type D with respect to $\lambda$ is the replacement of $D$ with $D' = D \setminus \{(i,i+1)\} \cup \{(i+2,i+3)\}$. If $\mu \subseteq \lambda$, then an \emph{excited diagram} of type D and shape $\lambda/\mu$ is a diagram contained in $[\lambda]^D$ that can be obtained from $[\mu]^D$ with a series of excited moves. Let $\p E^D(\lambda/\mu)$ denote the set of all excited diagrams of type D and shape $\lambda/\mu$. The following shows $\p E^D(432/2)$:

\medskip

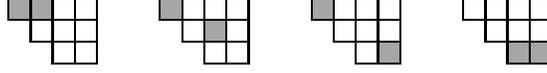
\begin{center}
\ytableausetup{smalltableaux}
 \begin{ytableau}
 *(gray!70) & *(gray!70) & & \\
  \none & & & \\
  \none & \none & &\\
\end{ytableau}
\qquad
\begin{ytableau}
 *(gray!70) &  & & \\
  \none & & *(gray!70) & \\
  \none & \none & &\\
\end{ytableau}
\qquad
\begin{ytableau}
 *(gray!70) & & & \\
  \none & & & \\
  \none & \none & & *(gray!70)\\
\end{ytableau}
\qquad
\begin{ytableau}
{} & & & \\
  \none & & & \\
  \none & \none & *(gray!70) & *(gray!70)\\
\end{ytableau}
\end{center}

\medskip

Recall from the introduction that Naruse's formula says that

\begin{equation*}
 f^{\lambda/\mu} = \frac{|\lambda/\mu|!}{\prod_{u \in [\lambda] } h^B(u)} \sum_{D \in \p E^B(\lambda/\mu)} \prod_{u \in  D} h^B(u) = \frac{|\lambda/\mu|!}{\prod_{u \in [\lambda] } h^D(u)} \sum_{D \in \p E^D(\lambda/\mu)} \prod_{u \in  D} h^D(u)
\end{equation*}
where all the hook lengths are evaluated in $[\lambda]^B$ (resp., $[\lambda]^D$).

\medskip

For example, the formula confirms that
\begin{multline*}
f^{432/2} = \frac{7!}{4 \cdot 7 \cdot 6 \cdot 3 \cdot 3 \cdot 5 \cdot 2 \cdot 2 \cdot 1} \left( 4 \cdot 7 + 4 \cdot 5 + 4 \cdot 1 + 3 \cdot 5 + 3 \cdot 1 + 2 \cdot 1 \right) \\
= \frac{7!}{7 \cdot 6 \cdot 4 \cdot 3 \cdot 5 \cdot 3 \cdot 2 \cdot 2 \cdot 1} \left( 7 \cdot 6 + 7 \cdot 3 + 7 \cdot 1 + 2 \cdot 1 \right) = 12
\end{multline*}

As we can see, the formulas are no longer equivalent.

\section{A polynomial identity} \label{identity}

It is clear that both sides of \eqref{eqB} and \eqref{eqD} are equal to $1$ if $\lambda = \mu$. Since the minimal entry of a shifted standard Young tableau of shape $\lambda/\mu$ must be in an outer corner of $\mu$ which lies in $\lambda$, we have $f^{\lambda/\mu} = \sum_{\mu \lessdot \nu \subseteq \lambda} f^{\lambda/\nu}$, where $\sum_{\mu \lessdot \nu \subseteq \lambda}$ denotes the sum over all strict partitions $\nu$ that are contained in $\lambda$ and cover $\mu$. If we show that the right-hand sides of \eqref{eqB} and \eqref{eqD} satisfy the same recursion, we are done. Therefore the statements are equivalent to the following identities:

\begin{align}
(|\lambda| - |\mu|) \sum_{D \in \p E^B(\lambda/\mu)} \prod_{u \in D} h^B(u) & = \sum_{\mu \lessdot \nu \subseteq \lambda} \sum_{D \in \p E^B(\lambda/\nu)} \prod_{u \in D} h^B(u). \label{eq2B} \\
(|\lambda| - |\mu|) \sum_{D \in \p E^D(\lambda/\mu)} \prod_{u \in D} h^D(u) & = \sum_{\mu \lessdot \nu \subseteq \lambda} \sum_{D \in \p E^D(\lambda/\nu)} \prod_{u \in D} h^D(u). \label{eq2D}
\end{align}

\begin{definition}\label{xi}
Define variables $x_i^B, x_i^D$, $i = 1,2,\ldots$, as follows:
\begin{itemize}
 \item if $i \leq \ell(\lambda)$, then $x_i^B = x_i^D$ is the number of cells in the $i$-th row of diagram $[\lambda]$; in other words, $x_i^B = x_i^D = \lambda_i$;
 \item if $i > \ell(\lambda)$, then $x_i^B$ is the number of cells in the $i$-th column of diagram $[\lambda]^B$, decreased by $i$, and $x_i^D$ is the number of cells in the $i$-th column of diagram $[\lambda]^D$ (which is the same as the number of cells in the $(i-1)$-th column of $[\lambda]^B$), decreased by $i-1$; in other words, $x_i^B = \lambda_{i}' - i$ and $x_i^D = \lambda_{i-1}'-i+1$.
\end{itemize}
\end{definition}

For example, for $\lambda = 432$, we have $x^B_1 = x^D_1 = 4$, $x^B_2 = x^D_2 = 3$, $x^B_3 = x^D_3 = 2$, $x^B_4 = -1$, $x^D_4 = 0$, $x^B_i = -i$ for $i \geq 5$, $x^D_5 = -1$, $x^D_i = -i+1$ for $i \geq -6$. It is clear that the sequences $x^B_1,x^B_2,\ldots$ and $x^D_1,x^D_2,\ldots$ are strictly decreasing. Also, $x^D_{\ell(\lambda) + 1}$ is always $0$.

\medskip

It is easy to see that for $(i,j) \in [\lambda]$, we have
\begin{align}
h^B_{ii} &= x^B_i,  & & 1 \leq i \leq \ell(\lambda) \label{hooks1} \\
h^B_{ij} &= x^B_i + x^B_j,  & & 1 \leq i \leq \ell(\lambda), \: i < j \leq \lambda_i + i - 1 \label{hooks2} \\
h^D_{ij} &= x^D_i + x^D_j,  & & 1 \leq i \leq \ell(\lambda), \: i < j \leq \lambda_i + i \label{hooks3}
\end{align}

\begin{definition} \label{ws}
 For arbitrary strict partitions $\lambda,\mu$, let $\p W^B(\mu,\lambda)$ consist of
\begin{itemize}
 \item $k \leq \ell(\lambda)$ for which $\lambda_k \neq \mu_i$ for all $i$,
 \item $k \geq \ell(\lambda) + 1$ for which $\lambda'_k - k \neq \mu'_i - i$ for all $i$,
\end{itemize}
 and let $\p W^D(\mu,\lambda)$ consist of
\begin{itemize}
 \item $k \leq \ell(\lambda)$ for which $\lambda_k \neq \mu_i$ for all $i$
 \item $\ell(\lambda) + 1$ if $\ell(\lambda) - \ell(\mu)$ is odd
 \item $k \geq \ell(\lambda) + 2$ for which $\lambda'_k - k \neq \mu'_i - i$ for all $i$
\end{itemize}
Furthermore, let $\p W(\lambda) = \{x_1,\ldots,x_{\lambda_1+1}\}$.
\end{definition}

A crucial role will also be played by the following lemma.

\begin{lemma} \label{w}
 For arbitrary strict partitions $\lambda,\mu$ (not necessarily satisfying $\mu \subseteq \lambda$) and $x^B_k$, $x^D_k$ defined as above (for $\lambda$), we have
 $$|\lambda| - |\mu| = \sum_{k \in \p W^B(\mu,\lambda)} x^B_k = \sum_{k \in \p W^D(\mu,\lambda)} x^D_k.$$
\end{lemma}
\begin{proof}
Since we have $\lambda'_k -k = \mu'_k -k = -k$ for large $k$, the sum in the lemma is finite. Now note the following: if $j$ appears as a part of a strict partition $\lambda$, $j = \lambda_i$, then $\lambda'_{i+j-1} \geq i$ and $\lambda'_{i+j} < i$. That means that $x^B_{i+j-1} = \lambda_{i+j-1}-i-j+1 \geq -j + 1$ and $x^B_{i+j} = \lambda_{i+j}-i-j < -j$. In other words, $-j$ does \emph{not} appear among $x^B_k$'s. On the other hand, if $j$ does not appear as a part of $\lambda$, there exists $i \geq 0$ so that $\lambda_{i+1} < j < \lambda_i$. Then $\lambda'_{i+j} = i$ and $x^B_{i+j} = i - i - j = -j$. We have proved that $\{\lambda_1,\ldots,\lambda_{\ell(\lambda)}\} = \{-x^B_k \colon k > \ell(\lambda)\}^C$, where the complement is taken with respect to the set of non-negative integers.\\
Clearly, $$|\lambda| - |\mu| = \sum_{k=1}^{\ell(\lambda)} \lambda_k - \sum_{k=1}^{\ell(\mu)} \mu_k = \sum_{\substack{k \leq \ell(\lambda) \\ \nexists i \colon \lambda_k = \mu_i}} \lambda_k - \sum_{\substack{k \leq \ell(\mu) \\ \nexists i \colon \mu_k = \lambda_i}} \mu_k.$$
Now
$$\{\mu_1,\ldots,\mu_{\ell(\mu)}\} \setminus \{\lambda_1,\ldots,\lambda_{\ell(\lambda)}\} = \{\lambda_1,\ldots,\lambda_{\ell(\lambda)}\}^C \setminus \{\mu_1,\ldots,\mu_{\ell(\mu)}\}^C = \{ -x^B_k \colon k > \ell(\lambda), \lambda'_k-k \neq \mu_i'-i\},$$
which proves the first statement. The proof for type D is very similar. As mentioned before, $x^D_{\ell(\lambda)+1} = 0$, so including it or not does not influence the sum.
\end{proof}

\begin{example}
 Take $\lambda = 432$ and $\mu = 2$. Since the parts $4$ and $3$ of $\lambda$ are not parts of $\mu$, while $2$ is, we have $1,2 \in \p W^B(432,2)$, $3 \notin \p W^B(432,2)$. Since
 $$(\lambda'_k-k)_{k \geq 4} = (-1,-5,-6,-7,\ldots), \qquad (\mu'_j-j)_{j \geq 1} = (0,-1,-2,-3,\ldots)$$
 none of the integers $4,5,\ldots$ are in $\p W^B(432,2)$. Indeed, $|\lambda| - |\mu| = x^B_1 + x^B_2 = 7$. We similarly get $|\lambda| - |\mu| = x^D_1 + x^D_2 = 7$.\\
 For $\lambda = 865321$ and $\mu = 431$, we have $\ell(\lambda)-\ell(\mu) = 3$,
 $$ (\lambda'_k-k)_{k \geq 7} = (-4,-7,-9,-10,-11,\ldots), \qquad (\mu'_j-j)_{j \geq 1} =  (0,0,0,-2,-5,-6,-7,\ldots)$$
 and $|\lambda| - |\mu| = x^B_1 + x^B_2 + x^B_3 + x^B_5 + x^B_7 = 8 + 6 + 5 + 2 + (-4) = 17$ and $|\lambda| - |\mu| = x^D_1 + x^D_2 + x^D_3 + x^D_5 + x^D_7 + x^D_8 = 8 + 6 + 5 + 2 + 0 + (-4) = 17$.
\end{example}

The following theorem is our main result. It gives two subtraction-free polynomial identities, which, by equations \eqref{hooks1}--\eqref{hooks3} and Lemma \ref{w}, specialize to equations \eqref{eq2B} and \eqref{eq2D} when $x_i = x_i^B$ (resp., $x_i = x_i^D$), and therefore imply the hook-length formulas for skew shifted diagrams.

\begin{theorem} \label{thm1}
For arbitrary strict partitions $\lambda$, $\mu$ and commutative variables $x_i$, we have
\begin{align*}
\left( \sum_{k \in \p W^B(\mu,\lambda)} x_k \right) \sum_{D \in \p E^B(\lambda/\mu)} \prod_{(i,i) \in D} x_i \prod_{\substack{(i,j) \in D \\ i < j}} (x_i + x_j) &= \sum_{\mu \lessdot \nu \subseteq \lambda} \sum_{D \in \p E^B(\lambda/\nu)} \prod_{(i,i) \in D} x_i \prod_{\substack{(i,j) \in D \\ i < j}} (x_i + x_j), \\
\left( \sum_{k \in \p W^D(\mu,\lambda)} x_k \right) \sum_{D \in \p E^D(\lambda/\mu)} \prod_{(i,j) \in D} (x_i + x_j) &= \sum_{\mu \lessdot \nu \subseteq \lambda} \sum_{D \in \p E^D(\lambda/\nu)} \prod_{(i,j) \in D} (x_i + x_j).
\end{align*}
\end{theorem}

The theorem is trivially true for $\mu \not\subseteq \lambda$, as then both sides of both equations are equal to $0$.

\begin{example}
 For $\lambda = 432$ and $\mu = 2$, we have the identities
 \begin{multline*}
 (x_1+x_2)(x_1 (x_1 + x_2) + x_1 (x_2 + x_3) + x_1 (x_3 + x_4) + x_2 (x_2 + x_3) + x_2 (x_3 + x_4) + x_3 (x_3 + x_4)) \\
 = x_1 (x_1 + x_2) (x_1 + x_3) + x_1 (x_1 + x_2) (x_2 + x_4) +  x_1 (x_2 + x_3) (x_2 + x_4) + x_2 (x_2 + x_3) (x_2 + x_4) \\ +  x_1 (x_1 + x_2) x_2 + x_1 (x_1 + x_2) x_3 +  x_1 (x_2 + x_3) x_3 + x_2 (x_2 + x_3) x_3\end{multline*}
 and
 \begin{multline*}
 (x_1+x_2) \left( (x_1 + x_2) (x_1 + x_3) + (x_1 + x_2) (x_2 + x_4) + (x_1 + x_2) (x_3 + x_5) + (x_3 + x_4) (x_3 + x_5) \right) \\
  = (x_1 + x_2) (x_1 + x_3) (x_2 + x_3) + (x_1 + x_2) (x_1 + x_3) (x_1 + x_4)\\
  + (x_1 + x_2) (x_1 + x_3) (x_2 + x_5) + (x_1 + x_2) (x_2 + x_4) (x_2 + x_5).
 \end{multline*}

 For $\lambda = 865321$ and $\mu = 431$, the first term on the left-hand side of the first equality is $x_1 + x_2 + x_3 + x_5 + x_7$, the second term is a sum of 4,992 monomials, and the right-hand side is a sum of 24,960 monomials. The first term on the left of the second equality is $x_1 + x_2 + x_3 + x_5 + x_7 + x_8$, the second term is a sum of 9,472 monomials, and the right-hand side is a sum of 56,832 monomials.
\end{example}

The (bijective) proof of Theorem \ref{thm1} is the content of the next section.

\section{The bijections} \label{bijection}

We interpret the two sides of equations \eqref{eqmainB} and \eqref{eqmainD} in terms of certain tableaux. To motivate the definition, look at the following excited diagrams of type B and D for $\lambda = 865321$ and $\mu = 431$:

\medskip

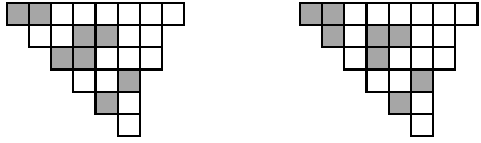
\begin{center}
   \ytableausetup{smalltableaux}
 \ytableausetup{aligntableaux=top}
 \begin{ytableau}
 *(gray!70) & *(gray!70) &  &  &  &  &  & \\
 \none & &  & *(gray!70) & *(gray!70) & &\\
 \none & \none & *(gray!70) & *(gray!70) & & & \\
 \none & \none & \none & & & *(gray!70) \\
 \none & \none & \none & \none & *(gray!70) & \\
 \none & \none & \none & \none & \none & \\
 \end{ytableau}
 \qquad\qquad
 \begin{ytableau}
 *(gray!70) & *(gray!70) &  &  &  &  &  & \\
 \none & *(gray!70)&  & *(gray!70) & *(gray!70) & &\\
 \none & \none & & *(gray!70) & & & \\
 \none & \none & \none & & & *(gray!70) \\
 \none & \none & \none & \none & *(gray!70) & \\
 \none & \none & \none & \none & \none & \\
 \end{ytableau}
\end{center}

\medskip

Instead of actually moving the cells of $\mu$, write an integer in a cell of $\mu$ that indicates how many times it moves (diagonally) from the original position. For the above example, we get the following tableaux of shape $\mu = 431$:

\medskip

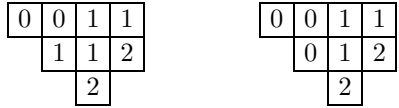
\begin{center}
   \ytableausetup{boxsize=1.25em}
 \ytableausetup{aligntableaux=top}
\begin{ytableau}
 0 & 0 & 1 & 1 \\
 \none & 1 & 1 & 2 \\
 \none & \none & 2
 \end{ytableau}
 \qquad\qquad
 \begin{ytableau}
 0 & 0 & 1 & 1 \\
 \none & 0 & 1 & 2 \\
 \none & \none & 2
 \end{ytableau}
\end{center}

\medskip

Rules for excited moves mean that the non-negative integer entries of the resulting tableau are weakly increasing along rows and columns. In type D, the entries on the diagonal must be even. Also, every tableau with non-negative integer entries and weakly increasing rows and columns corresponds to a valid excited diagram, provided that (in type D only) the entries on the diagonal are even, and the entry $r$ in row $i$ and column $j$ satisfies $j \leq \lambda_{i+r} + i - 1$ in type B and $j \leq \lambda_{i+r} + i$ in type D. Furthermore, it is enough to check this inequality only for the corners of $\mu$.

\medskip

The contributions
$$ \prod_{(i,i) \in D} x_i \prod_{\substack{(i,j) \in D \\ i < j}} (x_i + x_j), \qquad \prod_{(i,j) \in D} (x_i + x_j),$$
of type B/D excited diagram $D$ can be written as
$$ \prod_{(i,i) \in |\mu]^B} x_{i+T_{ii}} \prod_{\substack{(i,j) \in [\mu]^B \\ i < j}} (x_{i+T_{ij}} + x_{j+T_{ij}}), \qquad \prod_{(i,j) \in [\mu]^D} (x_{i+T_{ij}} + x_{j+T_{ij}}),$$
where $T$ is the corresponding tableau of shape $\mu$ with non-negative integer entries and weakly increasing rows and columns. To extract the monomials from the product, choose either $x_{i+T_{ij}}$ or $x_{j+T_{ij}}$ for each $(i,j)$ in $[\mu]^B$ or $[\mu]^D$, $i < j$, and multiply with $\prod_{(i,i) \in |\mu]^B} x_{i+T_{ii}}$ in type B. Write the number $T_{ij}$ in position $(i,j)$, $i < j$, in black if we choose $x_{i+T_{ij}}$, and in red if we choose $x_{j+T_{ij}}$. The number in $(i,i)$ in type B is always black. Call a tableau with non-negative integer black or red entries (red entries not allowed on the diagonal in type B, and odd entries not allowed on the diagonal in type D) and weakly increasing rows and columns a \emph{(shifted) bicolored tableau} of type B/D. Denote by $\p B^B(\mu)$ and $\p B^D(\mu)$ the (infinite unless $\mu = \emptyset$) set of shifted bicolored tableaux of type B/D and shape $\mu$, and denote by $\p B^B(\mu,\lambda)$ and $\p B^D(\mu,\lambda)$ the (finite) set of shifted bicolored tableaux $T$ of type B/D and shape $\mu$ that satisfy $j \leq \lambda_{i+T_{ij}} + i - 1$ in type B and $j \leq \lambda_{i+T_{ij}} + i$ in type D.

\medskip

The weight of a shifted bicolored tableau $T$ of shape $\mu$ is
\begin{align*}
 w^B(T) &= \prod_{(i,j) \in b(T)} x_{i + T_{ij}} \prod_{(i,j) \in [\mu]^B \setminus b(T)} x_{j + T_{ij}}\\
 w^D(T) &= \prod_{(i,j) \in b(T)} x_{i + T_{ij}} \prod_{(i,j) \in [\mu]^D \setminus b(T)} x_{j + T_{ij}}
\end{align*}
where $b(T)$ is the set of black entries of $T$.

\begin{example}
 The following are some bicolored tableaux in $\p B^B(431)$. A bicolored tableau is in $\p B^B(431,865321)$ if and only if $T_{25} \leq 2$, $T_{34} \leq 2$, so the first two are in $\p B^B(431,865321)$ and the third is not.

 \medskip

 \begin{center}
   \ytableausetup{boxsize=1.25em}
 \ytableausetup{aligntableaux=top}
 \ytableaushort
{0{\rd 0}0{\rd 1},{\none}001,{\none}{\none}{0}}
\qquad \qquad \ytableaushort
{{0}{\rd 0}11,{\none}11{\rd 2},{\none}{\none}2}
\qquad \qquad \ytableaushort
{{0}{\rd 0}{\rd 0}{\rd 2},{\none}{0}{1}{3},{\none}{\none}{2}}
\end{center}

\medskip

\noindent
The weights of these tableaux are $x_1^2x_2^3x_3^2x_5$, $x_1x_2^3x_3^2x_5x_6$, and $x_1x_2^2x_3^2x_5^2x_6$, respectively.\\
The following are some bicolored tableaux in $\p B^D(431)$. A bicolored tableau is in $\p B^D(431,865321)$ if and only if $T_{26} \leq 2$, $T_{35} \leq 2$, so the first two are in $\p B^D(431,765521)$ and the third is not.

 \medskip

 \begin{center}
   \ytableausetup{boxsize=1.25em}
 \ytableausetup{aligntableaux=top}
 \ytableaushort
{0{\rd 0}0{\rd 1},{\none}001,{\none}{\none}{\rd 0}}
\qquad \qquad \ytableaushort
{{\rd 0}{\rd 0}11,{\none}01{\rd 2},{\none}{\none}2}
\qquad \qquad \ytableaushort
{{0}{\rd 0}{\rd 0}{\rd 2},{\none}{0}{1}{3},{\none}{\none}{2}}
\end{center}

\medskip

\noindent
The weights of these tableaux are $x_1^2x_2^2x_3^2x_4x_6$, $x_2^4x_3^2x_5x_7$, and $x_1x_2x_3^2x_4x_5^2x_7$, respectively.
\end{example}

\medskip

The left-hand sides of equations \eqref{eqmainB} and \eqref{eqmainD} are the enumerators (with respect to weight $w$) of $\p B^B(\mu,\lambda) \times \p W^B(\mu,\lambda)$ and $\p B^D(\mu,\lambda) \times \p W^D(\mu,\lambda)$. The right-hand sides are the enumerators of the sets $\bigcup_{\nu} \p B^B(\nu,\lambda)$ and $\bigcup_{\nu} \p B^D(\nu,\lambda)$, where the union is over all strict partitions $\nu$ that cover $\mu$ and are contained in $\lambda$. In the remainder of this section, we present a weight-preserving bijection between the two sides in both cases.

\medskip

Let us first describe the \emph{bump of type B}. This is an algorithm that takes as input a shifted bicolored tableau $S$ of shape $\mu$, a cell $(i,j) \in [\mu] \cup \{(0,0)\}$, a direction $d \in \{ \downarrow, \rightarrow\}$, and a variable index $k$. Except when $i = j = 0$, we assume that $k$ can be in position $(i,j)$, either as a black $k-i$ if $d = \: \rightarrow$ or a red $k-j$ if $d = \: \downarrow$. As output, it produces a new tableau $S'$ (that differs from $S$ in at most one cell), a cell $(i',j')$ (which is either one column to the right or one row below $(i,j)$), a new direction $d' \in \{\downarrow, \rightarrow, \circ\}$, and a new variable index $k'$ ($\infty$ if $d' = \circ$). New direction being $\circ$ means that we must terminate the bumping process. Note that if $s = (0,0)$, $d$ will always be $\downarrow$. Without loss of generality, we can assume that $S$ fills the entire plane: write $\infty$ to the right and below the diagram, and $0$ above and to the left.

\medskip

If direction $d$ is $\rightarrow$, we move one column to the right; in other words, we set $j' = j+1$. Let $i'$ be the largest possible index so that we can insert variable $x_k$ as a black number $k-i'$ in position $(i',j')$; in other words, we must have $S_{i'-1,j'} \leq k - i' \leq S_{i'+1,j'}$, and $i'$ is the largest index with this property. Such $i'$ always exists and is at most $i$, as we will see in Section \ref{proofs}; hence $i' \leq i \leq j < j'$ and $(i',j')$ is never on the diagonal. It can happen that $(i',j')$ lies outside of $|\mu]$. In that case, $S'$ is the tableau we get if we write a black $k-i'$ in position $(i',j')$, $d' = \circ$, and $k = \infty$. If $(i',j') \in [\mu]$, $x_k$ bumps out whatever is in position $(i',j')$ as follows: $S'$ is the tableau we get if we replace the entry in position $(i',j')$ by a black $k-i'$. If $S_{i'j'}$ is black, the new direction $d'$ is $\rightarrow$, and the new variable index $k'$ is the entry in position $(i',j')$ of $S$, i.e.~$S_{i'j'} + i'$. If $S_{i'j'}$ is red, the new direction $d'$ is $\downarrow$ and the new variable index is $k' = S_{i'j'}+j'$.

\medskip

If direction $d$ is $\downarrow$, we move one row down; in other words, we set $i' = i+1$. Let $j' \geq i'$ be the largest possible index satisfying $S_{i',j'-1} \leq k - j' \leq S_{i',j'+1}$. It can happen that $(i',j')$ lies outside of $|\mu]$. In that case, $S'$ is the tableau we get if we write a red $k-j'$ in position $(i',j')$ (black $k-j' = k - i'$ if $i' = j' = \ell(\mu) +1$), $d' = \circ$, and $k = \infty$. If $(i',j') \in [\mu]$, $x_k$ bumps out whatever is in position $(i',j')$ as follows: $S'$ is the tableau we get if we replace the entry in position $(i',j')$ by a red $k-j'$ (black $k-j' = k - i'$ if $i' = j'$). If $S_{i'j'}$ is black, the new direction $d'$ is $\rightarrow$, and the new variable index $k'$ is the entry in position $(i',j')$ of $S$, i.e.~$S_{i'j'} + i'$. If $S_{i'j'}$ is red, the new direction $d'$ is $\downarrow$ and the new variable index is $k' = S_{i'j'}+j'$.

\medskip

In summary, a black number is bumped one column to the right and stays black, and a red number is bumped one row down and stays red unless it lands on the diagonal.

\medskip

The \emph{bump in type D} is very similar; now the diagonal entries can be red, but they must be even. Of course, in type D, the diagonal cell is of the form $(i,i+1)$, so a black entry in an even row gives an even index variable and a red entry in an even row gives an odd index variable (and the other way round for odd rows). So if we bump the index variable $k$ to the right or down and land on $(i,i+1)$, we write a black $k-i$ if that is an even number, and a red $k-i-1$ otherwise. Also, while red numbers are generally bumped down, a red number on the diagonal is bumped to the right. We leave out the details (but see a detailed example in Section \ref{proofs}).

\medskip

The \emph{insertion process} (in type B or D) takes as input a shifted bicolored tableau $S$ of appropriate type and shape $\mu$ and a variable index $k$. It first performs the bump (of type B or D) for $S$, $(i,j) = (0,0)$, $d = \: \downarrow$, and $k$, with output $S'$, $(i',j')$, $d'$, and $k'$. If $d' = \circ$, the insertion process returns $S'$; otherwise, it performs the bump for $S'$, $(i',j')$, $d'$, and $k'$, with output $S''$, $(i'',j'')$, $d''$ and $k''$. It continues in this way until the direction is $\circ$.

\medskip

Finally, the \emph{repeated insertion process} (in type B or D) takes as input a bicolored tableau in $\p B^B(\mu,\lambda)$ or $\p B^D(\mu,\lambda)$ and a variable index $k$ in $\p W^B(\mu,\lambda)$ or $\p W^D(\mu,\lambda)$. It performs the insertion process with $S$ and $k$. If the resulting tableau $S'$ is in $\p B(\lambda,\nu)$ for $\nu$, $\mu \lessdot \nu \subseteq \lambda$, the repeated insertion process outputs $S'$. Otherwise, let $k'$ be the variable index represented by the entry in $[\nu] \setminus [\mu]$. Perform the insertion process with $S'$ and $k'$ as inputs. If the resulting tableau $S''$ is in $\p B(\lambda,\nu)$ for $\nu$, $\mu \lessdot \nu \subseteq \lambda$, the repeated insertion process outputs $S''$. Otherwise, we continue in this manner.

\medskip

It is not obvious that the bump, the insertion process, and the repeated insertion process are well defined. However, we will sketch the proof of the following theorem in Section \ref{proofs}.

\begin{theorem} \label{welldefined}
 The indices $i'$, $j'$ and $k'$ in the bump of type B or D are well defined, and $S'$ is a well-defined bicolored tableau.\\
 The insertion processes of type B and D described above always terminate and are weight-preserving bijections
 $$\psi^B_\mu \colon \p B^B(\mu) \times \{x_1,x_2,x_3,\ldots \} \longrightarrow \bigcup_{\nu} \p B^B(\nu),$$
 $$\psi^D_\mu \colon \p B^D(\mu) \times \{x_1,x_2,x_3,\ldots \} \longrightarrow \bigcup_{\nu} \p B^D(\nu),$$
 where the unions are over all strict partitions $\nu$ which cover $\mu$.\\
 The repeated insertion processes of type B and D described above always terminate and are weight-preserving bijections
 $$\Psi^B_{\mu,\lambda} \colon \p B^B(\mu,\lambda) \times \p W^B(\mu,\lambda) \longrightarrow \bigcup_{\nu} \p B^B(\nu,\lambda),$$
 $$\Psi^D_{\mu,\lambda} \colon \p B^D(\mu,\lambda) \times \p W^D(\mu,\lambda) \longrightarrow \bigcup_{\nu} \p B^D(\nu,\lambda),$$
 where the unions are over all strict partitions $\nu$ which cover $\mu$ and are contained in $\lambda$.
\end{theorem}

This theorem proves \eqref{eqmainB} and \eqref{eqmainD} and hence the hook-length formulas for skew shapes, equations \eqref{eqB} and \eqref{eqD}.

\medskip

While both the bump and the insertion process are efficient, repeated insertion can be very slow. Indeed, in Section \ref{proofs}, we prove the following.

\begin{theorem} \label{complexity}
 For every positive integer $m$, there exists a shifted bicolored tableau $S$ of type B and shape $\mu = (m+2,m,m-1,m-2\ldots,1)$ so that repeated insertion of $2$ into $S$ with respect to $\lambda = (m+2,m+1,m-1,m-2\ldots,1)$ needs $2^m$ insertions.
\end{theorem}

It is worthwhile to consider what happens when $\ell(\lambda) = \ell(\mu)$. In that case, $\lambda/\mu$ can be interpreted as a skew non-shifted shape, and we can compute $f^{\lambda/\mu}$ using Naruse's formula for straight shapes. Of course, the formulas \eqref{eq2B} and \eqref{eq2D} give an equivalent result: we can have no excited moves on or above the diagonal, so the contribution of these cells on both sides can be canceled.
\ref{w}
\section{A weighted generalization of the hook-length formula} \label{weighted}

After reading \cite{kon:naruse}, Darij Grinberg (personal communication) pointed out that a certain weighted gen\-er\-al\-i\-za\-tion (\cite[Theorem 18]{hopkins}) of the hook-length formula has a Naruse-style extension to skew shapes, and that it can be proved by taking a different specialization of the main polynomial identity from \cite{kon:naruse}. It turns out that the skew shifted shapes also allow for this generalization.

\medskip

Recall that the \emph{content} of a cell $(i,j)$ in a Young diagram of a strict partition $\lambda$ is the integer $j-i$, which we denote $c_\lambda(i,j)$ (or simply $c(i,j)$ if it is clear what partition we are considering). Note that in type B, the diagonal cells have content $0$, while in type D, they have content $1$. Let $z_k$, $k \in \Z_{\geq 0}$, be some commutative variables. Define weighted hook lengths $h^B(u;{\bf z})$ and $h^D(u;{\bf z})$ of a cell $u$ in a diagram of $\lambda$ of type B/D as the sum of $z_{c_\lambda(i,j)}$ over all cells $(i,j)$ in the type B/D hook of $u$. For example, the weighted hook lengths of the cells $(2,2)$, $(1,4)$ and $(2,5)$ in $[6532]^B$ are $z_0 + z_1 + z_2 + z_3 + z_4 + z_5$, $2z_0 + 2z_1 + z_2 + z_3 + z_4 + z_5$, and $z_1 + z_2 + z_3 + z_4$. On the other hand, the weighted hook lengths of the cells $(1,4)$ and $(2,5)$ in $[6532]^D$ are $2z_1 + 2z_2 + z_3 + z_4 + z_5 + z_6$ and $z_1 + z_2 + z_3 + z_4 + z_5$. Of course, if $z_k = 1$ for all $k$, the weighted hook length becomes the usual hook length.

\medskip

Furthermore, for a standard Young tableau $T$ of skew shifted shape $\lambda/\mu$, $|\lambda|-|\mu| = n$, define
$$T^B_{\bf z} = \prod_{k=1}^n \frac 1{z_{c_\lambda(T^{-1}(n))} + z_{c_\lambda(T^{-1}(n-1))} + \cdots + z_{c_\lambda(T^{-1}(k))}},$$
where $T^{-1}(k)$ is the unique cell of the tableau with entry $k$, and we take contents in type B. We define $T^D_{\bf z}$ analogously (and we obtain it from $T^B_{\bf z}$by shifting all indices by $1$). If $z_k = 1$ for all $k$, $T^B_{\bf z} = T^D_{\bf z} = 1/n!$. For example, if $T$ is the tableau

$$\begin{ytableau}
 *(gray!70) & *(gray!70) & 2 & 3 \\
\none & 1 & 4 & 5 \\
\none & \none & 6 & 7
\end{ytableau}$$

then

$$T^B_{\bf z} = \frac 1{(2z_0+2z_1+2z_2+z_3)(z_0+2z_1+2z_2+z_3)(z_0+2z_1+z_2+z_3)(z_0+2z_1+z_2)(z_0+z_1+z_2)(z_0 + z_1)z_1},$$
$$T^D_{\bf z} = \frac 1{(2z_1+2z_2+2z_3+z_4)(z_1+2z_2+2z_3+z_4)(z_1+2z_2+z_3+z_4)(z_1+2z_2+z_3)(z_1+z_2+z_3)(z_1 + z_2)z_2}.$$

\begin{theorem} \label{z}
  For a strict partition $\lambda$, we have
  \begin{align}
   \sum_{T \in \SYT(\lambda/\mu)} T^B_{\bf z} &= \sum_{D \in \p E^B(\lambda/\mu)} \prod_{u \in [\lambda]^B \setminus D} \frac 1{h^B(u;\bf z)} = \frac{1}{\prod_{u \in [\lambda]^B } h^B(u;\bf z)} \sum_{D \in \p E^B(\lambda/\mu)} \prod_{u \in  D} h^B(u;\bf z) \\
   \sum_{T \in \SYT(\lambda/\mu)} T^D_{\bf z}&= \sum_{D \in \p E^D(\lambda/\mu)} \prod_{u \in [\lambda]^D \setminus D} \frac 1{h^D(u;\bf z)} = \frac{1}{\prod_{u \in [\lambda]^D } h^D(u;\bf z)} \sum_{D \in \p E^D(\lambda/\mu)} \prod_{u \in  D} h^D(u;\bf z)
  \end{align}
\end{theorem}
\begin{proof}[Sketch of proof]
 We only sketch the proof for type B; type D and all details are left as an exercise for the reader. the term in $T^B_{\bf z}$ corresponding to $k = 1$ equals $1/(\sum_{u \in [\lambda/\mu]^B} z_{c_\lambda(u)}))$ and is therefore independent of the tableau $T$ (it depends only on $\lambda$ and $\mu$). That means that the equivalent recursive form of the first equality is
 $$\sum_{u \in [\lambda/\mu]^B} z_{c_\lambda(u)}) \sum_{D \in \p E^B(\lambda/\mu)} \prod_{u \in D} h^B(u;\bf z) = \sum_{\mu \lessdot \nu \subseteq \lambda} \sum_{D \in \p E^B(\lambda/\nu)} \prod_{u \in D} h^B(u;\bf z).$$
 It turns out that this is just another specialization of the first polynomial identity in Theorem \ref{thm1}.\\
 Define $x_i^B(\bf z) = z_0 + \ldots + z_{\lambda_i-1}$ for $i \leq \ell(\lambda)$ and  $x_i^B(\bf z) = - z_0 - \ldots - z_{i-\lambda'_i-1}$ for $i > \ell(\lambda)$. Note that when $z_k = 1$ for all $k$, $x_i^B(\bf z) = x_i^B$. It is easy to check that $h^B(i,i;\bf z) = x_i^B(\bf z)$ for $i \leq \ell(\lambda)$, and $h^B(i,j;\bf z) = x_i^B(\bf z) + x_j^B(\bf z)$ for $i \leq \ell(\lambda)$, $i < j \leq \lambda_i + i - 1$. Furthermore, it is a straightforward generalization of Lemma \ref{w} that
 $$\sum_{u \in [\lambda/\mu]^B} z_{c_\lambda(u)} = \sum_{k \in \p W^B(\mu,\lambda)} x^B_k(\p z).$$
 The result follows.
\end{proof}

\begin{example}
 For $\lambda = 321$ and $\mu = 1$, the theorem says that
 \begin{multline*}
   \frac{1}{z_0 (z_0+z_1) (2 z_0+z_1) (2 z_0+z_1+z_2) (2 z_0+2 z_1+z_2)} \\ + \frac{1}{z_0 (z_0+z_1) (z_0+z_1+z_2) (2 z_0+z_1+z_2) (2 z_0+2 z_1+z_2)} \\
  = \frac{1}{z_0 (z_0+z_1) (2 z_0+z_1) (z_0+z_1+z_2) (2 z_0+z_1+z_2) (2 z_0+2 z_1+z_2)} \cdot \Big((z_0+z_1+z_2) + (z_0+z_1) + z_0\Big).
 \end{multline*}
 Note that the terms on the left correspond to standard Young tableaux
 \begin{center}
 \ytableausetup{centertableaux}
 \ytableausetup{smalltableaux}
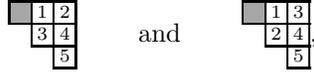
 \begin{ytableau}
  *(gray!70) & 1 & 2 \\
 \none & 3 & 4 \\
 \none & \none & 5
\end{ytableau} \qquad \mbox{and} \qquad
 \begin{ytableau}
  *(gray!70) & 1 & 3 \\
 \none & 2 & 4 \\
 \none & \none & 5
 \end{ytableau},
\end{center}
while the three terms in the sum on the right correspond to the excited diagrams
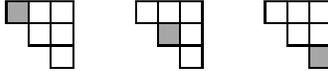
\begin{center}
\ytableausetup{smalltableaux}
 \begin{ytableau}
 *(gray!70) & & \\
  \none & & \\
  \none & \none &\\
\end{ytableau}
\qquad
\begin{ytableau}
{} & & \\
 \none & *(gray!70) & \\
 \none & \none &\\
\end{ytableau}
\qquad
\begin{ytableau}
{} & & \\
 \none & & \\
 \none & \none & *(gray!70)\\
\end{ytableau}
\qquad
\end{center}
\end{example}

Finally, let us mention that an interesting specialization of Theorem \ref{z} is $z_i = q^i$, which gives a $q$-version of the skew shifted hook-length formula.

\section{Examples and proofs} \label{proofs}

\subsection*{Examples of bumps, insertion, and repeated insertion}

As our first example in type B, take $\lambda = 865321$, $\mu = 431$, and perform repeated insertion with the tableau
$$\ytableausetup{boxsize=1.25em}
 \ytableausetup{aligntableaux=top}
 \ytableaushort
{0{\rd 0}{\rd 1}{\rd 1},{\none}{1}2{2},{\none}{\none}{2}}$$
and $k = 1$. Note that $\lambda_1 = 8$ is not a part of $\mu$, so $1$ is indeed in $\p W^B(\mu,\lambda)$.

\medskip

By the definition of the insertion process, we first take $(i,j) = (0,0)$ and $d = \: \downarrow$. Because direction is downward, we move one row down, to row $i' = 1$, and find the largest $j'$ so that the number $1-j'$ can be inserted in position $(1,j')$. Of course, the only option is $j' = 1$, and because $(1,1)$ is on the diagonal, we insert a black $0$ in that position. A black $0$ is also bumped out, so the new direction is $\rightarrow$, and the new variable index is $1$. We move one column to the right. The only possible position is $(1,2)$, where we insert a black $0$, and bump the variable $x_2$ (represented by the red $0$) downward. It lands in position $(2,2)$ and bumps out the black $1$. We move to the right. The variable $x_3$ that was bumped can either be represented as a black $2$ in position $(3,1)$ or a black $2$ in position $(3,2)$. Both are possible, so we choose the latter option. A black $2$ is bumped to the right, land in $(2,4)$ and bumps out a black $2$. This one lands in position $(1,5)$ as a black $3$. The bumps are represented in the following figure. The gray cells show the positions of bumped cells.

\medskip

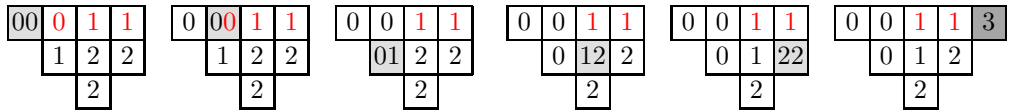
\begin{center}
\ytableausetup{boxsize=1.25em}
\ytableausetup{aligntableaux=top}
\ytableaushort
{{*(gray!25) 00}{\rd 0}{\rd 1}{\rd 1},{\none}{1}2{2},{\none}{\none}{2}}
\quad \ytableaushort
{{0}{*(gray!25) 0 \rd 0}{\rd 1}{\rd 1},{\none}{1}2{2},{\none}{\none}{2}}
\quad \ytableaushort
{{0}{0}{\rd 1}{\rd 1}{\none},{\none}{*(gray!25)01}2{2},{\none}{\none}{2}}
\ytableaushort
{{0}{0}{\rd 1}{\rd 1},{\none}{0}{*(gray!25)12}{2},{\none}{\none}{2}}
\quad \ytableaushort
{{0}{0}{\rd 1}{\rd 1},{\none}{0}{1}{*(gray!25)22},{\none}{\none}{2}}
\quad \ytableaushort
{{0}{0}{\rd 1}{\rd 1}{*(gray!70) 3},{\none}{0}{1}{2},{\none}{\none}{2}}
\end{center}

\medskip

Since $(1,5) \notin [431]^B$, the insertion process is done. On the other hand, the repeated insertion process is not, since the entry in position $(1,5)$ violates condition $\leq 2$. Therefore we remove this cell, and insert the variable $x_4$ into the resulting tableau. The insertion process is represented in the following figure.

\medskip

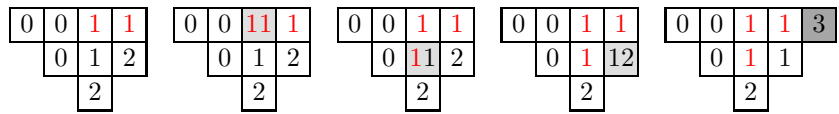
\begin{center}
   \ytableausetup{boxsize=1.25em}
 \ytableausetup{aligntableaux=top}
 \ytableaushort
 {{0}{0}{\rd 1}{\rd 1},{\none}{0}{1}{2},{\none}{\none}{2}}
 \quad\ytableaushort
 {{0}{0}{*(gray!25) \rd 1\rd 1}{\rd 1},{\none}{0}{1}{2},{\none}{\none}{2}}
 \quad\ytableaushort
 {{0}{0}{\rd 1}{\rd 1},{\none}{0}{*(gray!25) {\rd 1}1}{2},{\none}{\none}{2}}
 \quad\ytableaushort
 {{0}{0}{\rd 1}{\rd 1},{\none}{0}{\rd 1}{*(gray!25) 12},{\none}{\none}{2}}
 \quad\ytableaushort
 {{0}{0}{\rd 1}{\rd 1}{*(gray!70) 3},{\none}{0}{\rd 1}{1},{\none}{\none}{2}}
\end{center}

\medskip

Again, the repeated insertion process is not finished, since the entry in position $(1,5)$ is $> 2$. We remove the cell, and insert $x_4$ into the remaining tableau. The insertion process is now as follows:

\medskip

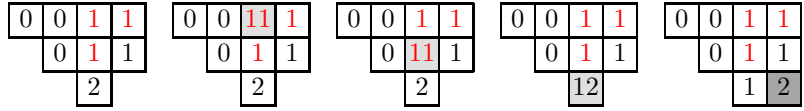
\begin{center}
   \ytableausetup{boxsize=1.25em}
 \ytableausetup{aligntableaux=top}
 \ytableaushort
{{0}{0}{\rd 1}{\rd 1},{\none}{0}{\rd 1}{1},{\none}{\none}{2}}
\quad \ytableaushort
{{0}{0}{*(gray!25) \rd 1 \rd 1}{\rd 1},{\none}{0}{\rd 1}{1},{\none}{\none}{2}}
\quad \ytableaushort
{{0}{0}{\rd 1}{\rd 1},{\none}{0}{*(gray!25) \rd 1 \rd 1}{1},{\none}{\none}{2}}
\quad \ytableaushort
{{0}{0}{\rd 1}{\rd 1},{\none}{0}{\rd 1}{1},{\none}{\none}{*(gray!25) 12}}
\quad \ytableaushort
{{0}{0}{\rd 1}{\rd 1},{\none}{0}{\rd 1}{1},{\none}{\none}{1}{*(gray!70) 2}}
\end{center}

\medskip

The last shifted bicolored tableau is in $\p B^B(431,865321)$, so we are done with the repeated insertion process.

\medskip

Let us mention that if we perform the repeated insertion process on all 24,960 pairs $(S,k)$ in
$$\p B^B(431,865321) \times \p W^B(431,865321),$$
the process needs just one insertion in 17,398 cases, two insertions in 6,080 cases, three insertions in 977 cases, four insertions in 455 cases, five insertions in 25 cases, and six insertions in 25 cases.

\medskip

Based on this, one might conjecture that the number of insertions in the repeated insertion process is bounded by $|\mu|$, $\ell(\lambda)$ or something similar. The next example shows that this is not the case. Take $\lambda = 65321$, $\mu = 64321$, $k = 2$, and the tableau

\medskip

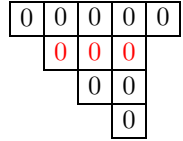
\begin{center}
  \ytableausetup{boxsize=1.25em}
 \ytableausetup{aligntableaux=top}
 \ytableaushort
{{0}{0}{0}{0}{0}, {\none}{\rd 0}{\rd 0}{\rd 0}, {\none}{\none}{0}{0}, {\none}{\none}{\none}{0}}
\end{center}

\medskip

Now we represent each insertion in one diagram, with the gray cells indicating the insertion path. We see that we need 16 insertions. In fact, we will use a generalization of this example to construct an example with $\lambda$ and $\mu$ of (approximately) size $m^2/2$ and length $m$, and $(S,k) \in \p B^B(\mu,\lambda) \times \p W^B(\mu,\lambda)$ that needs $2^m$ steps of the insertion process.

\medskip

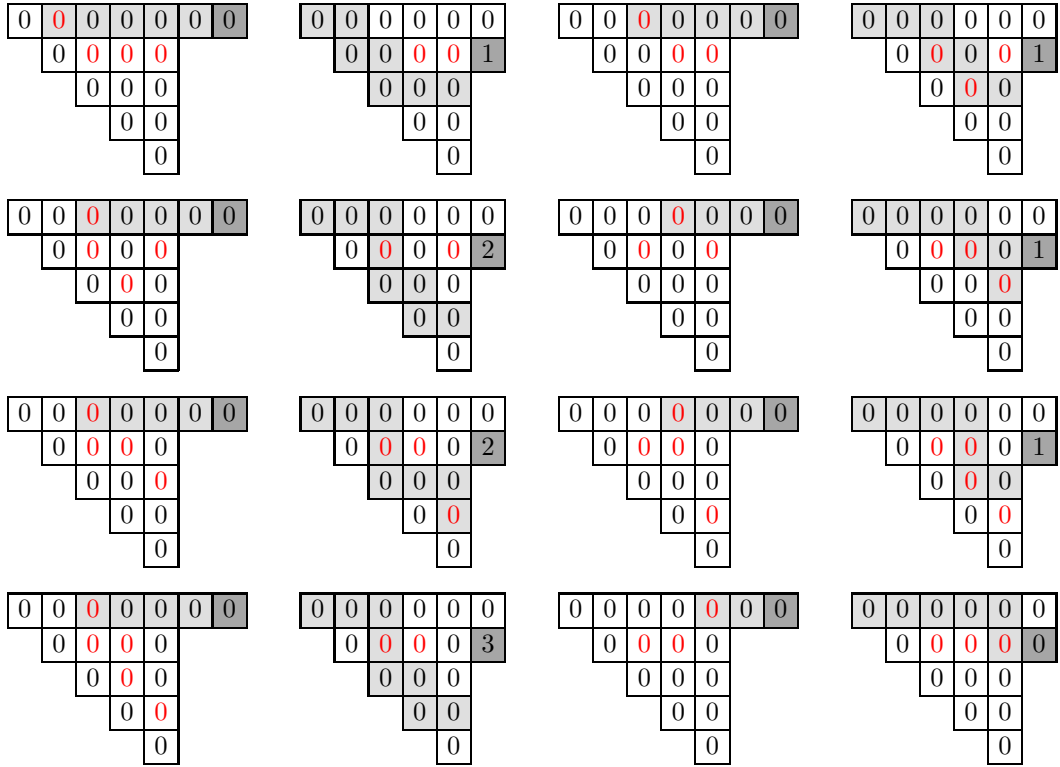
\begin{center}
  \ytableausetup{boxsize=1.25em}
 \ytableausetup{aligntableaux=top}
 \ytableaushort
{{0}{*(gray!25) \rd 0}{*(gray!25) 0}{*(gray!25) 0}{*(gray!25) 0}{*(gray!25) 0}{*(gray!70)0}, {\none}{0}{\rd 0}{\rd 0}{\rd 0}, {\none}{\none}{0}{0}{0}, {\none}{\none}{\none}{0}{0}, {\none}{\none}{\none}{\none}{0}}
\qquad \ytableaushort
{{*(gray!25)0}{*(gray!25) 0}{0}{0}{0}{0}, {\none}{*(gray!25)0}{*(gray!25) 0}{\rd 0}{\rd 0}{*(gray!70)1}, {\none}{\none}{*(gray!25)0}{*(gray!25) 0}{*(gray!25) 0}, {\none}{\none}{\none}{0}{0}, {\none}{\none}{\none}{\none}{0}}
\qquad \ytableaushort
{{0}{0}{*(gray!25) \rd 0}{*(gray!25) 0}{*(gray!25) 0}{*(gray!25) 0}{*(gray!70)0}, {\none}{0}{0}{\rd 0}{\rd 0}, {\none}{\none}{0}{0}{0}, {\none}{\none}{\none}{0}{0}, {\none}{\none}{\none}{\none}{0}}
\qquad \ytableaushort
{{*(gray!25)0}{*(gray!25) 0}{*(gray!25) 0}{0}{0}{0}, {\none}{0}{*(gray!25) \rd 0}{*(gray!25) 0}{\rd 0}{*(gray!70)1}, {\none}{\none}{0}{*(gray!25) \rd 0}{*(gray!25) 0}, {\none}{\none}{\none}{0}{0}, {\none}{\none}{\none}{\none}{0}}

\vspace{0.3cm}
\ytableaushort
{{0}{0}{*(gray!25) \rd 0}{*(gray!25) 0}{*(gray!25) 0}{*(gray!25) 0}{*(gray!70)0}, {\none}{0}{\rd 0}{0}{\rd 0}, {\none}{\none}{0}{\rd 0}{0}, {\none}{\none}{\none}{0}{0}, {\none}{\none}{\none}{\none}{0}}
\qquad \ytableaushort
{{*(gray!25)0}{*(gray!25) 0}{*(gray!25) 0}{0}{0}{0}, {\none}{0}{*(gray!25) \rd 0}{0}{\rd 0}{*(gray!70)2}, {\none}{\none}{*(gray!25)0}{*(gray!25) 0}{0}, {\none}{\none}{\none}{*(gray!25)0}{*(gray!25) 0}, {\none}{\none}{\none}{\none}{0}}
\qquad \ytableaushort
{{0}{0}{0}{*(gray!25) \rd 0}{*(gray!25) 0}{*(gray!25) 0}{*(gray!70)0}, {\none}{0}{\rd 0}{0}{\rd 0}, {\none}{\none}{0}{0}{0}, {\none}{\none}{\none}{0}{0}, {\none}{\none}{\none}{\none}{0}}
\qquad \ytableaushort
{{*(gray!25)0}{*(gray!25) 0}{*(gray!25) 0}{*(gray!25) 0}{0}{0}, {\none}{0}{\rd 0}{*(gray!25) \rd 0}{*(gray!25) 0}{*(gray!70)1}, {\none}{\none}{0}{0}{*(gray!25) \rd 0}, {\none}{\none}{\none}{0}{0}, {\none}{\none}{\none}{\none}{0}}

\vspace{0.3cm}
\ytableaushort
{{0}{0}{*(gray!25) \rd 0}{*(gray!25) 0}{*(gray!25) 0}{*(gray!25) 0}{*(gray!70)0}, {\none}{0}{\rd 0}{\rd 0}{0}, {\none}{\none}{0}{0}{\rd 0}, {\none}{\none}{\none}{0}{0}, {\none}{\none}{\none}{\none}{0}}
\qquad \ytableaushort
{{*(gray!25)0}{*(gray!25) 0}{*(gray!25) 0}{0}{0}{0}, {\none}{0}{*(gray!25) \rd 0}{\rd 0}{0}{*(gray!70)2}, {\none}{\none}{*(gray!25)0}{*(gray!25) 0}{*(gray!25) 0}, {\none}{\none}{\none}{0}{*(gray!25) \rd 0}, {\none}{\none}{\none}{\none}{0}}
\qquad \ytableaushort
{{0}{0}{0}{*(gray!25) \rd 0}{*(gray!25) 0}{*(gray!25) 0}{*(gray!70)0}, {\none}{0}{\rd 0}{\rd 0}{0}, {\none}{\none}{0}{0}{0}, {\none}{\none}{\none}{0}{\rd 0}, {\none}{\none}{\none}{\none}{0}}
\qquad \ytableaushort
{{*(gray!25)0}{*(gray!25) 0}{*(gray!25) 0}{*(gray!25) 0}{0}{0}, {\none}{0}{\rd 0}{*(gray!25) \rd 0}{0}{*(gray!70)1}, {\none}{\none}{0}{*(gray!25) \rd 0}{*(gray!25) 0}, {\none}{\none}{\none}{0}{\rd 0}, {\none}{\none}{\none}{\none}{0}}

\vspace{0.3cm}
\ytableaushort
{{0}{0}{*(gray!25) \rd 0}{*(gray!25) 0}{*(gray!25) 0}{*(gray!25) 0}{*(gray!70)0}, {\none}{0}{\rd 0}{\rd 0}{0}, {\none}{\none}{0}{\rd 0}{0}, {\none}{\none}{\none}{0}{\rd 0}, {\none}{\none}{\none}{\none}{0}}
\qquad \ytableaushort
{{*(gray!25)0}{*(gray!25) 0}{*(gray!25) 0}{0}{0}{0}, {\none}{0}{*(gray!25) \rd 0}{\rd 0}{0}{*(gray!70)3}, {\none}{\none}{*(gray!25)0}{*(gray!25) 0}{0}, {\none}{\none}{\none}{*(gray!25)0}{*(gray!25) 0}, {\none}{\none}{\none}{\none}{0}}
\qquad \ytableaushort
{{0}{0}{0}{0}{*(gray!25) \rd 0}{*(gray!25) 0}{*(gray!70)0}, {\none}{0}{\rd 0}{\rd 0}{0}, {\none}{\none}{0}{0}{0}, {\none}{\none}{\none}{0}{0}, {\none}{\none}{\none}{\none}{0}}
\qquad \ytableaushort
{{*(gray!25)0}{*(gray!25) 0}{*(gray!25) 0}{*(gray!25) 0}{*(gray!25) 0}{0}, {\none}{0}{\rd 0}{\rd 0}{*(gray!25) \rd 0}{*(gray!70)0}, {\none}{\none}{0}{0}{0}, {\none}{\none}{\none}{0}{0}, {\none}{\none}{\none}{\none}{0}}
\end{center}

\medskip

Let us show an example in type D. Again, take $\lambda = 865321$, $\mu = 431$, and insert variable $x_1$ (i.e.~$k=1$) into the tableau
$$\ytableausetup{boxsize=1.25em}
 \ytableausetup{aligntableaux=top}
 \ytableaushort
{{\rd 0}{\rd 0}{\rd 0}{\rd 2},{\none}{0}{\rd 1}{2},{\none}{\none}{\rd 2}}$$

Note that $\lambda_1 = 8$ is not a part of $\mu$, so $1$ is indeed in $\p W^D(\mu,\lambda)$.

\medskip

By the definition of the insertion process, we first take $(i,j) = (0,0)$ and $d = \: \downarrow$. Because direction is downward, we move one row down, to row $i' = 1$, and try to insert variable $x_1$ as a red entry. However, in row $j' \geq 2$, $x_1$ is represented by a red $1-j'$, which is always negative. It is also odd for $j' = 2 = i'+1$, so we can insert it (only) as a black $k-i' = 0$ in position $(i',j') = (1,2)$. The index variable that is bumped out is $k' = 2$ (represented by the red $0$ in column $2$). Even though the entry in $(i',j')$ is red, we bump it to the right (because $j' = i'+1$). In other words, $d' = \: \rightarrow$. The variable $x_2$ can be represented as a black $1$ in row $1$ (impossible, as then the entry in position $(1,3)$ would be larger than the one in position $(2,3)$) or a black $0$ in row $2$. We choose the latter option. In other words, $(i'',j'') = (2,3)$, and the variable bumped out is $x_2$ (note that we just saw that it can happen that the bumped entry remains unchanged), which gets bumped to the right. It lands in position $(2,4)$ and bumps out the red $1$, representing the variable $x_5$. This variable travels downward; it can be represented as a black $2$ in position $(3,4)$ (not red $1$, since odd numbers are forbidden on the diagonal!) or red $0$ in position $(3,5)$. The latter is not allowed, as the entry in $(3,5)$ cannot be smaller than the one in $(3,4)$. So the variable bumps out the red $2$. Since the red $2$, representing the variable $x_6$, was on the diagonal, it goes to the right, and lands in position $(3,5)$ as a black $3$. The bumps are represented in the following figure. The gray cells show the positions of bumped cells.

\medskip

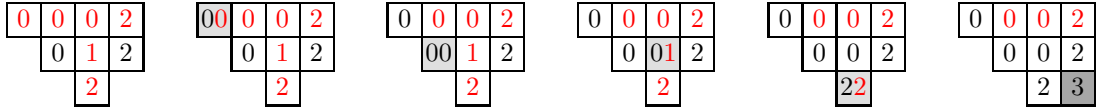
\begin{center}
   \ytableausetup{boxsize=1.25em}
 \ytableausetup{aligntableaux=top}
 \ytableaushort
{{\rd 0}{\rd 0}{\rd 0}{\rd 2},{\none}{0}{\rd 1}{2},{\none}{\none}{\rd 2}}
\qquad \ytableaushort
{{*(gray!25) 0 \rd 0}{\rd 0}{\rd 0}{\rd 2},{\none}{0}{\rd 1}{2},{\none}{\none}{\rd 2}}
\qquad \ytableaushort
{{0}{\rd 0}{\rd 0}{\rd 2},{\none}{*(gray!25) 00}{\rd 1}{2},{\none}{\none}{\rd 2}}
\qquad \ytableaushort
{{0}{\rd 0}{\rd 0}{\rd 2},{\none}{0}{*(gray!25) 0 \rd 1}{2},{\none}{\none}{\rd 2}}
\qquad \ytableaushort
{{0}{\rd 0}{\rd 0}{\rd 2},{\none}{0}{0}{2},{\none}{\none}{*(gray!25) 2 \rd 2}}
\qquad \ytableaushort
{{0}{\rd 0}{\rd 0}{\rd 2},{\none}{0}{0}{2},{\none}{\none}{2}{*(gray!70) 3}}
\end{center}

\medskip

Since $(3,5) \notin [431]^D$, the insertion process is done. On the other hand, the repeated insertion process is not, since the entry in position $(3,5)$ violates condition $\leq 2$. Therefore we remove this cell, and insert the variable $x_6$ into the resulting tableau. The insertion process is represented in the following figure.

\medskip

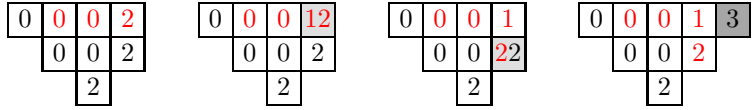
\begin{center}
   \ytableausetup{boxsize=1.25em}
 \ytableausetup{aligntableaux=top}
 \ytableaushort
{{0}{\rd 0}{\rd 0}{\rd 2},{\none}{0}{0}{2},{\none}{\none}{2}}
\qquad \ytableaushort
{{0}{\rd 0}{\rd 0}{*(gray!25) \rd 1 \rd 2},{\none}{0}{0}{2},{\none}{\none}{2}}
\qquad \ytableaushort
{{0}{\rd 0}{\rd 0}{\rd 1},{\none}{0}{0}{*(gray!25) \rd 2 2},{\none}{\none}{2}}
\qquad \ytableaushort
{{0}{\rd 0}{\rd 0}{\rd 1}{*(gray!70) 3},{\none}{0}{0}{\rd 2},{\none}{\none}{2}}
\end{center}

\medskip

Again, the repeated insertion process is not finished, since the entry in position $(1,6)$ is $> 2$. We remove the cell, and insert $x_4$ into the remaining tableau. The insertion process is now as follows:

\medskip

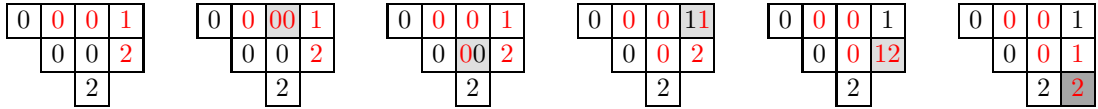
\begin{center}
   \ytableausetup{boxsize=1.25em}
 \ytableausetup{aligntableaux=top}
 \ytableaushort
{{0}{\rd 0}{\rd 0}{\rd 1},{\none}{0}{0}{\rd 2},{\none}{\none}{2}}
\qquad \ytableaushort
{{0}{\rd 0}{*(gray!25) \rd 0 \rd 0}{\rd 1},{\none}{0}{0}{\rd 2},{\none}{\none}{2}}
\qquad \ytableaushort
{{0}{\rd 0}{\rd 0}{\rd 1},{\none}{0}{*(gray!25) \rd 0 0}{\rd 2},{\none}{\none}{2}}
\qquad \ytableaushort
{{0}{\rd 0}{\rd 0}{*(gray!25) 1 \rd 1},{\none}{0}{\rd 0}{\rd 2},{\none}{\none}{2}}
\qquad \ytableaushort
{{0}{\rd 0}{\rd 0}{1},{\none}{0}{\rd 0}{*(gray!25) \rd 1 \rd 2},{\none}{\none}{2}}
\qquad \ytableaushort
{{0}{\rd 0}{\rd 0}{1},{\none}{0}{\rd 0}{\rd 1},{\none}{\none}{2}{*(gray!70) \rd 2}}
\end{center}

\medskip

The last shifted bicolored tableau is in $\p B^D(431,865321)$, so we are done with the repeated insertion process.

\medskip

If we perform the repeated insertion process on all 56,832 pairs $(S,k)$ in
$$\p B^D(431,865321) \times \p W^D(431,865321),$$
the process needs just one insertion in 42,672 cases, two insertions in 11,087 cases, three insertions in 2,182 cases, four insertions in 741 cases, five insertions in 88 cases, and six insertions in 62 cases.

\medskip

We can easily adapt the second example for type B to show that repeated insertion process in type D may also need many steps. Note also that the complexity does not come from the fact that the diagram is shifted: there are also examples with many steps for the repeated insertion process in \cite{kon:naruse}.

\subsection*{Properties of the bump, the insertion process, and the repeated insertion process}

Say that we have a tableau $S$, $(i,j)$, $d$ and $k$, and that $d = \: \rightarrow$. The bump algorithm in type B says that we should take $j' = j + 1$, and find the largest possible $i'$ so that we can write $k - i'$ in position $(i',j')$ while keeping the column $j'$ weakly increasing. Since the sequence $(S_{i'j'})_{i' \geq 0}$ is weakly increasing, $(S_{i'j'} + i')_{i' \geq 0}$ is strictly increasing. There are two options. One is that $k = S_{i''j'} + i''$ for some $i''$. In that case, we have just one possible choice for $i'$, and that is $i''$. The entry $S_{i'j'}$ remains unchanged, so the new tableau is still a valid shifted bicolored tableau. Also, we cannot have $i' > i$, since then $k - i' \geq S_{ij'}$, and that contradicts the assumption that we can write $k - i > k - i'$ in position $(i,j)$ in $S$. If, on the other hand, we have $S_{i''-1,j'} + i'' - 1 < k < S_{i''j'} + i''$ for some $i''$, then $S_{i''-1,j'} < k - (i''-1) \leq S_{i''j'}$ and $S_{i''-1,j'} \leq k - i'' < S_{i''j'}$. So we can either write $k - (i''-1)$ in position $(i''-1,j')$ or $k - i''$ in position $(i'',j')$. We select $i' = i''$. Now the entry in position $(i',j')$ becomes smaller, and we have to check that $S_{i',j'-1} = S_{i'j} \leq k - i'$. Again, we cannot have $i' > i$, so the fact that $k-i$ can be written in position $(i,j)$ means that $k - i' \geq k - i \geq S_{i'j}$.

\medskip

The analysis when  $d = \: \downarrow$ is very similar and we leave it as an exercise for the reader. In summary, the process either moves one column to the right and weakly up, or one row down and weakly to the left, and numbers in the tableau can only get smaller.

\medskip

In order to prove that the insertion process terminates, it is enough to prove that a certain (integer) quantity with an upper bound increases at each step. We claim that such a quantity is
$$\left\{ \begin{array}{ccl} k + j & \colon & d = \: \rightarrow \\ k + i & \colon & d = \: \downarrow \end{array}\right.$$
When we go to the right, $j' = j + 1$, and when we go down, $i' = i + 1$. In each case we bump a weakly larger number, so $k' \geq k$. It is also clear that this quantity is bounded by $\ell(\mu) + \mu_1 + \max T_{ij}$.

\medskip

The inverse of the bump is easy to construct, and repeated application of it gives the inverse of the insertion process. The fact that the bump (and consequently $\psi_\mu$) is weight preserving is obvious. We omit the details for type D bump and insertion.


\medskip

The fact that the repeated insertion process terminates (and is a weight-preserving bijection) can be proved in a very similar way as the proof of \cite[Theorem 11]{kon:naruse} using sieve equivalence. More precisely, in order for that lemma to apply and complete the proof of Theorem \ref{thm1}, we need the following result.

\begin{lemma}
 Let $\psi_\mu(T,k)$ denote the result of type B insertion of variable index $k$ into the tableau $T \in \p B^B(\mu)$. If $S \in \p B^B(\nu)$ for $\mu \lessdot \nu$, $T$ is the tableau of shape $\mu$ that we obtain if we remove the unique cell of $S$ outside of $\mu$, and $k$ is the variable index of the removed entry, then write $\varphi_\mu(S) = (T,k)$. Denote the (finite) set $\p B^B(\mu,\lambda) \times \p W(\lambda)$ by $A$, its image under $\psi_\mu$ by $B$, its subset $\p B^B(\mu,\lambda) \times \p W^B(\mu,\lambda)$ by $X$, and the set $\bigcup_\nu \p B^B(\nu,\lambda)$ by $Y$. Then the restriction of $\varphi_\mu$ to $B \setminus Y$ is an injection with image $A \setminus X$.
\end{lemma}

The proof is analogous to the proof of \cite[Lemma 12]{kon:naruse}, and proves that the repeated insertion process is a weight-preserving bijection that completes the proof of the main theorem for type B. Again, we omit the details for type D.

\subsection*{Complexity of the algorithm} Let $m \geq 1$. Let $S^{m,0}$ be the shifted bicolored tableau of type B and shape $\mu^{(m)} = (m+2,m,m-1,m-2,\ldots,1)$ with black $0$'s everywhere except in columns $3,\ldots,m+1$ of row $2$, where it has red $0$'s, let $k_0 = 2$ and $\lambda^{(m)} = (m+2,m+1,m-1,m-2,\ldots,1)$. Our second example in Section \ref{proofs} is the repeated insertion of $(S^{4,0},k_0)$. The following is defined for $m \geq 1$, $n \geq 0$, and is clearly a generalization of the definition of $S^{m,0}$:

$$S^{m,n}_{ij} = \left\{ \begin{array}{ccl}
0 & \colon & i = j \\
{\rd 0} & \colon & j > i = 1, \, n \mbox{ odd and } (n + 1)/2^{j-2} \in \{2,3,5,7,9,11,\ldots\} \\
0 & \colon & j > i = 1, \,  n \mbox{ even or } (n + 1)/2^{j-2} \notin \{2,3,5,7,9,11,\ldots\} \\
0 &\colon & j > i = 2, \,  j = 1 + \mbox{number of digits in binary expansion of } n \\
{\rd 0} &\colon & j > i = 2, \,  j \neq 1 + \mbox{number of digits in binary expansion of } n \\
{\rd 0} & \colon & j > i \geq 3, \,  (\lfloor \frac{n+2}{2^{i-2}} \rfloor - 1)/2^{j-i} \mbox{ odd positive integer}\\
0 & \colon & j > i \geq 3, \,  (\lfloor \frac{n+2}{2^{i-2}} \rfloor - 1)/2^{j-i} \mbox{ not odd positive integer}\end{array} \right.$$
$$k_n = \left\{ \begin{array}{ccl} 1 & \colon & n \mbox{ odd} \\ j+1 & \colon & n = 2^j - 2 \\ j + 2 & \colon & n \mbox{ even, }(n+2)/2^j \mbox{ odd integer} \end{array} \right.$$

\begin{prop}
 For $0 \leq n < 2^m$, the insertion process of type B with input $(S^{m,n},k_n)$ ouputs the tableau obtained from $S^{m,n+1}$ by adding a black $k_{n+1} - 1 = 0$ in position $(1,m+3)$ if $n$ is even, and a black $k_{n+1} - 2$ in position $(2,m+2)$ if $n$ is odd, except when $n = 2^m - 1$, where the output has a black $0$ in position $(2,m+2)$. Therefore the repeated insertion process for $(S^{m,0},k_0)$ stops after precisely $2^m$ steps.
\end{prop}

The proof is easy for $n$ even, and a tedious book-keeping exercise for $n$ odd. We leave it is an exercise for the reader. It is clear that the proposition proves Theorem \ref{complexity}.

\section*{Acknowledgments}

The author would like to thank Alejandro Morales, Igor Pak and Greta Panova for telling him about the problem and the interesting discussions that followed. Many thanks to Darij Grinberg for suggesting the appropriate weighted generalization and other useful comments, and to Sara Billey for fruitful talks about the subject.

\bibliographystyle{siam}

\def\cprime{$'$}

\end{document}